\documentclass[reqno,11pt]{amsart}

\usepackage[a4paper,margin=30mm]{geometry}
\usepackage{amssymb,amsmath,amsthm,mathtools}
\usepackage{microtype}

\theoremstyle{plain}
\newtheorem{theorem}{Theorem}[section]

\newtheorem{proposition}[theorem]{Proposition}
\newtheorem{corollary}[theorem]{Corollary}

\theoremstyle{definition}
\newtheorem{definition}[theorem]{Definition}

\theoremstyle{remark}
\newtheorem{remark}[theorem]{Remark}

\newcommand{\elltwo}{\ell_2}

\newcommand{\N}{\mathbb{N}}
\newcommand{\K}{\mathbb{K}}
\renewcommand{\le}{\leqslant}
\renewcommand{\ge}{\geqslant}

\title[Raja's covering index of $L_p$ spaces]%
{Raja's covering index of $L_p$ spaces}

\author[T.~Kania]{Tomasz Kania}
\address[T.~Kania]{Mathematical Institute\\Czech Academy of Sciences\\\v Zitn\'a 25 \\115 67 Praha 1\\Czech Republic and  Institute of Mathematics and Computer Science\\ Jagiellonian University\\ {\L}ojasiewicza 6, 30-348 Krak\'{o}w, Poland}
\email{kania@math.cas.cz, tomasz.marcin.kania@gmail.com}

\author[N.~Ma\'slany]{Natalia Ma\'slany}
\address[N.~Ma\'slany]{Institute of Mathematics and Computer Science  \\ Jagiellonian University\\ {\L}ojasiewicza 6, 30-348 Krak\'{o}w, Poland}
\email{nataliamaslany97@gmail.com}


\thanks{RVO: 67985840. The second-named author acknowledged with thanks funding received from NCN Sonata-Bis 13 (2023/50/E/ST1/00067). }

\begin{document}
\maketitle

\begin{abstract}
We study Raja's covering index $\Theta_X(n)$ for classical $L_p$-spaces and their non-commutative counterparts.  
For infinite-dimensional Hilbert spaces we compute the covering index exactly, proving
\[
\Theta_H(n)=n^{-1/2}\qquad(n\in\mathbb N);
\]
in particular $\Theta_H(2)=1/\sqrt2$, thus answering a question of Raja about the precise two-piece covering index of $\elltwo$.  
For scalar-valued Lebesgue spaces $L_p(\mu)$, $1\le p<\infty$, we construct an explicit block decomposition of the unit ball yielding the upper bound $\Theta_{L_p(\mu)}(n)\le n^{-1/p}$ for all $n\in\mathbb{N}$; in particular $\Theta_{\ell_p}(n)\le n^{-1/p}$.  
For $1<p<\infty$, under the corresponding $p$-AUS renormability hypothesis, this combines with Raja's general lower bound to give the sharp asymptotic estimate $\Theta_{L_p(\mu)}(n)\asymp n^{-1/p}$.  

We also obtain uniform upper bounds $\Theta_{L_p(\mu;E)}(n)\le n^{-1/p}$ for Bochner spaces $L_p(\mu;E)$ over non-atomic $\sigma$-finite measure spaces, with constants independent of the Banach space $E$; this shows that, at the level of power-type upper estimates, the covering index decays at the same rate regardless of the asymptotic geometry of~$E$ and provides a partial negative answer to a problem of Raja.  
Finally, using non-commutative Clarkson inequalities, we derive power-type lower bounds $\Theta_{L_p(M,\tau)}(n)\gtrsim n^{-1/r}$ for non-commutative $L_p(M,\tau)$ spaces associated with semifinite von Neumann algebras, where $r=\min\{p,2\}$.  We do not attempt to optimise the exponent or constants in the non-commutative setting.
\end{abstract}

\section*{Introduction}

In a recent preprint \cite{RajaCoveringIndex}, Raja introduced a new quantitative invariant of Banach spaces, the \emph{covering index} $\Theta_X(n)$, defined via convex coverings of the unit ball by finitely many pieces each containing a large finite-codimensional ball.  More precisely, for a Banach space $X$ and each $n\in\mathbb{N}$, the quantity $\Theta_X(n)$ is the infimum, over all coverings of $B_X$ by $n$ closed convex sets, of the maximum \emph{essential inradius} among those sets.  This invariant measures how resistant the unit ball is to being decomposed into convex parts, and turns out to be closely related to asymptotic uniform smoothness and the Szlenk index.

The covering index provides a natural quantitative refinement of classical decomposition techniques in Banach space theory.  While various notions of asymptotic geometry---such as asymptotic uniform smoothness, asymptotic uniform convexity, and the Szlenk index---have been extensively studied and applied to problems in renorming theory, the structure of reflexive spaces, and fixed-point theory, the covering index offers a fresh perspective by encoding geometric information through optimal decompositions of the unit ball.  In particular, the decay rate of $\Theta_X(n)$ as $n\to\infty$ captures subtle structural features of~$X$ that are not immediately apparent from standard moduli alone.  Raja's framework establishes that the covering index sits naturally between local and global geometric properties: on the one hand, it can be estimated via asymptotic smoothness moduli; on the other hand, its precise value depends on the interplay between finite-dimensional and infinite-dimensional features of the space.

Raja developed a general framework connecting the decay of $\Theta_X(n)$ as $n\to\infty$ with the asymptotic geometry of $X$.  In particular, he proved that if $X$ admits an equivalent $p$-asymptotically uniformly smooth ($p$-AUS) norm, then one necessarily has a power-type lower bound
\[
\Theta_X(n)\ \gtrsim\ n^{-1/p}.
\]
Raja computed or estimated $\Theta_X$ for several concrete classes of spaces and posed a number of open problems about exact values in specific cases.  One of these asks for the precise value of $\Theta_{\elltwo}(2)$; more generally, while his methods identify the correct exponent for many classical sequence spaces, they do not yield sharp constants.  Another question (Problem~4.7 in \cite{RajaCoveringIndex}) asks to what extent the behaviour of $\Theta_X$ is governed by the asymptotic moduli of uniform convexity and uniform smoothness of $X$.\smallskip

The purpose of this note is threefold.\smallskip

First, we compute the covering index of infinite-dimensional Hilbert spaces exactly:
for every infinite-dimensional Hilbert space $H$ and every $n\in\mathbb N$,
\[
\Theta_H(n)=n^{-1/2}.
\]
In particular,
\[
\Theta_H(2)=\frac1{\sqrt2}.
\]
Raja explicitly asked for the exact value of $\Theta_{\ell_2}(2)$
(Problem~4.6 in~\cite{RajaCoveringIndex}).
The upper bound is elementary and comes from an orthogonal decomposition of $H$
into $n$ infinite-dimensional subspaces, while the matching lower bound follows
from an explicit computation of Raja's goal derivation on Hilbert balls.

Second, we give an explicit block decomposition of the unit ball in the scalar-valued Lebesgue spaces $L_p(\mu)$, $1\le p<\infty$, which produces coverings by $n$ closed convex sets of essential inradius $n^{-1/p}$.  This yields the upper bound
\[
\Theta_{L_p(\mu)}(n)\ \le\ n^{-1/p}\qquad(1\le p<\infty,\ n\in\mathbb{N}),
\]
and, in particular, $\Theta_{\ell_p}(n)\le n^{-1/p}$ for the sequence spaces.  For $1<p<\infty$, whenever $L_p(\mu)$ admits an equivalent $p$-AUS norm (as is the case, for instance, for the sequence spaces $\ell_p$ and for $L_p[0,1]$; see \cite{KnaustOdellSchlumprecht1999,RajaAUS}), Raja's general theorem supplies a matching lower bound of order $n^{-1/p}$.  We therefore obtain the sharp asymptotic behaviour
\[
\Theta_{L_p(\mu)}(n)\ \asymp\ n^{-1/p}
\qquad(1<p<\infty),
\]
and a completely explicit covering which exhibits the optimal decay rate.

Third, we use the same block-decomposition technique to obtain upper bounds in the vector-valued setting and to comment on Raja's Problem~4.7.  For any non-zero Banach space~$E$ and any non-atomic $\sigma$-finite measure space $(\Omega,\Sigma,\mu)$, we show that the covering index of the Bochner space $L_p(\mu;E)$ satisfies
\[
\Theta_{L_p(\mu;E)}(n)\ \le\ n^{-1/p}
\qquad(1\le p<\infty,\ n\in\mathbb{N}),
\]
with no dependence on the geometry of~$E$.  Since the asymptotic moduli of $L_p(\mu;E)$ can vary dramatically as $E$ ranges over separable Banach spaces (including spaces that are not AUSable at all), this shows that, at least on the level of power-type upper bounds, the covering index decays at the same rate regardless of the asymptotic uniform smoothness of~$E$; this provides a partial negative answer to Problem~4.7 of \cite{RajaCoveringIndex}.

Finally, we briefly discuss non-commutative $L_p$-spaces.  If $(M,\tau)$ is a semifinite von Neumann algebra with faithful normal semifinite trace, then the associated non-commutative spaces $L_p(M,\tau)$, $1<p<\infty$, are uniformly convex and uniformly smooth.  Non-commutative Clarkson inequalities show that their modulus of smoothness has power type $r=\min\{p,2\}$ \cite{PisierXu}, which in turn implies that $L_p(M,\tau)$ is $r$-AUS for $1<p<\infty$.  Applying Raja's general theorem we obtain power-type lower estimates
\[
\Theta_{L_p(M,\tau)}(n)\ \gtrsim\ n^{-1/r},
\qquad r=\min\{p,2\},
\]
for the covering index of non-commutative $L_p$-spaces.  Even in the commutative case these bounds are not optimal when $p>2$, and we do not attempt to determine the sharp exponent or constants in the non-commutative setting.

\section{Preliminaries}\label{sec:Preliminaries}

Let $X$ be a Banach space with unit ball $B_X$.

\begin{definition}[Essential inradius {\cite{RajaCoveringIndex}}]
For a set $A\subset X$, define
\[
\varrho(A)
=
\sup\Bigl\{r>0: \exists x\in A\
   \exists Y\le X\ \text{closed}, \dim X/Y < \infty, \dim Y = \infty,\
   x+r(B_X\cap Y)\subset A\Bigr\}.
\]
\end{definition}

\begin{definition}[Covering index {\cite{RajaCoveringIndex}}]
Let $\mathcal{C}(X)$ denote the family of closed, convex, bounded subsets of $X$.
For $n\in\mathbb{N}$ we define the \emph{covering index} of $X$ at level $n$ by
\[
\Theta_X(n)
=
\inf\Bigl\{
   \max_{1\le k\le n} \varrho(A_k)
   \;:\;
   A_k\in\mathcal{C}(X)\ (k=1,\dots,n),\ 
   B_X = \bigcup_{k=1}^n A_k
\Bigr\}.
\]
\end{definition}

The index is \emph{isomorphic}: equivalent norms change $\Theta_X$ only up to absolute multiplicative constants \cite[Sec.~1]{RajaCoveringIndex}.

For completeness, we record the standard AUS definitions and Raja's general lower bound.

\begin{definition}
Let $X$ be a Banach space. For $t\ge0$, we define the \emph{(global) modulus of smoothness}
\[
\rho_X(t)
=
\sup_{\|x\|=1,\ \|y\|=1}
\left(\frac{1}{2}\|x+ty\| + \frac{1}{2}\|x-ty\| - 1\right),
\]
and the \emph{asymptotic modulus of smoothness}
\[
\overline{\rho}_X(t)
=
\sup_{x\in S_X}\ 
\inf_{\substack{Y\le X\\ \dim(X/Y)<\infty}}\ 
\sup_{y\in S_Y}(\|x+ty\|-1).
\]
We say that $X$ is \emph{asymptotically uniformly smooth} (AUS) if
$\overline{\rho}_X(t)=o(t)$ as $t\to0$.
For $p\in(1,\infty)$ we say that $X$ is \emph{$p$-AUS} if
$\overline{\rho}_X(t)\le C t^p$ for some $C>0$ and all $t>0$.
\end{definition}

\begin{remark}\label{rmrk:modulus}
For every Banach space $X$ and every $t>0$ one has
\[
\overline{\rho}_X(t)\le \rho_X(t).
\]
\end{remark}

\begin{definition}[$p$-AUSable space]
A Banach space $X$ is \emph{$p$-AUSable} if it admits an~equivalent $p$-AUS norm.
\end{definition}

\begin{theorem}[Raja's lower bound {\cite[Thm.~3.5]{RajaCoveringIndex}}]\label{thm:Raja-lower}
If $X$ is $p$-AUSable, then there exists $c>0$ such that
\[
\Theta_X(n)\ \ge\ c\,n^{-1/p}
\qquad(n\ge1).
\]
\end{theorem}

For $1<p<\infty$, the classical spaces $\ell_p$ and $L_p(\mu)$ are uniformly
smooth, and their usual modulus of smoothness is of power type
$r=\min\{p,2\}$. In the scalar-valued setting we shall moreover use the known
fact that $L_p(\mu)$ admits an equivalent $p$-AUS norm in the situations under
consideration. Thus Theorem~\ref{thm:Raja-lower} yields lower bounds of the
form $\Theta_{\ell_p}(n)\gtrsim n^{-1/p}$ and, under the corresponding
renormability hypothesis, $\Theta_{L_p(\mu)}(n)\gtrsim n^{-1/p}$. Our aim is
to complement these with explicit upper bounds of the same order in the
scalar-valued setting.

\section{The Hilbert space}

Let $H$ be an infinite-dimensional Hilbert space.

For the lower bound we use Raja's \emph{goal derivation}; see
\cite[Sec.~3]{RajaCoveringIndex}.  If $A\subset H$ is weakly compact and convex
and $\varepsilon>0$, define
\[
[A]'_{\varepsilon}
:=
\bigl\{
x\in A:\ \text{for every weak neighbourhood $U$ of $x$, one has }
\varrho(A\cap U)>\varepsilon
\bigr\}.
\]
Set
\[
[A]^0_{\varepsilon}:=A,
\qquad
[A]^{m+1}_{\varepsilon}:=\bigl[[A]^m_{\varepsilon}\bigr]'_{\varepsilon}
\qquad(m\in\N\cup\{0\}).
\]
Following Raja, the corresponding goal Szlenk index is defined by
\[
Gz(A,\varepsilon)
:=
\inf\{m\in\N:\ [A]^m_{\varepsilon}=\varnothing\},
\]
where $\inf\varnothing:=\infty$.
Hence Corollary~3.2 of \cite{RajaCoveringIndex} implies that, for every Banach
space $X$ and every $n\in\N$,
\begin{equation}\label{eq:goal-empty}
\varepsilon>\Theta_X(n)\quad\Longrightarrow\quad [B_X]^n_{\varepsilon}=\varnothing.
\end{equation}

We first compute the goal derivation explicitly on Hilbert balls.

\begin{proposition}\label{prop:Hilbert-goal}
Let $r>0$ and $\varepsilon>0$. Then
\[
[(rB_H)]'_{\varepsilon}
=
\begin{cases}
\sqrt{r^2-\varepsilon^2}\,B_H,& 0<\varepsilon<r,\\[1mm]
\varnothing,& \varepsilon\ge r.
\end{cases}
\]
\end{proposition}

\begin{proof}
First suppose that $\varepsilon\ge r$. We claim that
$[(rB_H)]'_{\varepsilon}=\varnothing$.
Indeed, let $x\in rB_H$ and let $U$ be any weak neighbourhood of $x$.
If
\[
z+t(B_H\cap Y)\subset rB_H\cap U
\]
for some $z\in rB_H\cap U$, some closed finite-codimensional infinite-dimensional
subspace $Y\le H$, and some $t>0$, then we may choose
$y\in S_H\cap Y\cap z^{\perp}$, since $Y\cap z^\perp$ is still infinite-dimensional.
Thus $z+ty\in rB_H$, and so
\[
\|z\|^2+t^2=\|z+ty\|^2\le r^2.
\]
Hence $t\le r\le\varepsilon$, which shows that
$\varrho(rB_H\cap U)\le\varepsilon$. Therefore $x\notin[(rB_H)]'_{\varepsilon}$.
As $x$ was arbitrary, $[(rB_H)]'_{\varepsilon}=\varnothing$.

Now assume $0<\varepsilon<r$, and set
\[
s:=\sqrt{r^2-\varepsilon^2}.
\]
We prove that
\[
[(rB_H)]'_{\varepsilon}=sB_H.
\]

\smallskip
\noindent\emph{Inclusion $sB_H\subset[(rB_H)]'_{\varepsilon}$.}
Let $x\in H$ with $\|x\|\le s$, and let $U$ be a weak neighbourhood of $x$.
Choose a basic weak neighbourhood
\[
V
=
\bigl\{
z\in H:\ |\langle z-x,u_i\rangle|<\delta,\ i=1,\dots,m
\bigr\}
\subset U
\]
for suitable $u_1,\dots,u_m\in H$ and $\delta>0$.

If $\|x\|<s$, choose $t$ so that
\[
\varepsilon<t<\sqrt{r^2-\|x\|^2}.
\]
Let
\[
F:=\operatorname{span}\{x,u_1,\dots,u_m\},
\qquad
Y:=F^\perp.
\]
Then $Y$ is closed, finite-codimensional and infinite-dimensional.  For every
$y\in B_H\cap Y$ we have
\[
\langle x+ty-x,u_i\rangle=\langle ty,u_i\rangle=0
\qquad(i=1,\dots,m),
\]
so $x+ty\in V$.  Also, since $x\in F$ and $y\in F^\perp$,
\[
\|x+ty\|^2=\|x\|^2+t^2\|y\|^2\le \|x\|^2+t^2<r^2,
\]
hence $x+t(B_H\cap Y)\subset rB_H\cap V\subset rB_H\cap U$.  Therefore
$\varrho(rB_H\cap U)\ge t>\varepsilon$.

If $\|x\|=s$, choose $\lambda\in(0,1)$ so close to $1$ that
$x_0:=\lambda x\in V$.  Then $\|x_0\|<\|x\|=s$, so we may choose
\[
\varepsilon<t<\sqrt{r^2-\|x_0\|^2}.
\]
With
\[
F:=\operatorname{span}\{x,u_1,\dots,u_m\},
\qquad
Y:=F^\perp,
\]
the same argument gives
\[
x_0+t(B_H\cap Y)\subset rB_H\cap V\subset rB_H\cap U,
\]
and again $\varrho(rB_H\cap U)>\varepsilon$.
Thus $x\in[(rB_H)]'_{\varepsilon}$.

\smallskip
\noindent\emph{Inclusion $[(rB_H)]'_{\varepsilon}\subset sB_H$.}
Let $x\in H$ with $\|x\|>s$.  Set $e:=x/\|x\|$, and choose $\alpha$ with
\[
s<\alpha<\|x\|.
\]
Consider the weakly open set
\[
U:=\bigl\{z\in H:\ \operatorname{Re}\langle z,e\rangle>\alpha\bigr\}.
\]
Then $x\in U$.

We claim that $\varrho(rB_H\cap U)\le\varepsilon$.  Suppose, towards a contradiction,
that $\varrho(rB_H\cap U)>\varepsilon$.  Then there exist
$z\in rB_H\cap U$, a closed finite-codimensional infinite-dimensional subspace
$Y\le H$, and $t>\varepsilon$ such that
\[
z+t(B_H\cap Y)\subset rB_H\cap U.
\]
Choose $y\in S_H\cap Y\cap z^\perp$.  Since $z+ty\in rB_H$, we get
\[
\|z\|^2+t^2=\|z+ty\|^2\le r^2.
\]
But $z\in U$ implies
\[
\|z\|\ge \operatorname{Re}\langle z,e\rangle>\alpha,
\]
and therefore
\[
t<\sqrt{r^2-\alpha^2}<\sqrt{r^2-s^2}=\varepsilon,
\]
a contradiction.  Hence $\varrho(rB_H\cap U)\le\varepsilon$, so
$x\notin[(rB_H)]'_{\varepsilon}$.

This proves $[(rB_H)]'_{\varepsilon}=sB_H$.
\end{proof}

\begin{corollary}\label{cor:Hilbert-goal-iterates}
For every $\varepsilon>0$ and every $m\in\N$,
\[
[B_H]^m_{\varepsilon}
=
\begin{cases}
\sqrt{1-m\varepsilon^2}\,B_H,& m\varepsilon^2<1,\\[1mm]
\varnothing,& m\varepsilon^2\ge1.
\end{cases}
\]
\end{corollary}

\begin{proof}
This follows immediately from Proposition~\ref{prop:Hilbert-goal} by induction on $m$.
\end{proof}

We can now compute the covering index of $H$ exactly.

\begin{theorem}\label{thm:Hilbert-n}
For every $n\in\N$,
\[
\Theta_H(n)=n^{-1/2}.
\]
In particular,
\[
\Theta_H(2)=\frac1{\sqrt2}.
\]
\end{theorem}

\begin{proof}
\emph{Upper bound.}
Choose an orthogonal decomposition
\[
H=H_1\oplus\cdots\oplus H_n
\]
into infinite-dimensional closed subspaces, and let $P_j$ denote the orthogonal
projection onto $H_j$.  For $j=1,\dots,n$, define
\[
A_j:=\Bigl\{x\in B_H:\ \|P_jx\|\le n^{-1/2}\Bigr\}.
\]
Each $A_j$ is closed, convex and bounded.  Since
\[
\sum_{j=1}^n \|P_jx\|^2=\|x\|^2\le1
\qquad(x\in B_H),
\]
at least one of the numbers $\|P_jx\|$ is at most $n^{-1/2}$, and therefore
\[
B_H=\bigcup_{j=1}^n A_j.
\]

If $\|y\|\le n^{-1/2}$, then $\|P_jy\|\le \|y\|\le n^{-1/2}$, so
$y\in A_j$. Hence $\varrho(A_j)\ge n^{-1/2}$.

Conversely, suppose that $\varrho(A_j)>n^{-1/2}$. Then there exist
$x\in A_j$, a closed finite-codimensional infinite-dimensional subspace
$Y\le H$, and $r>n^{-1/2}$ such that
\[
x+r(B_H\cap Y)\subset A_j.
\]
Since $Y\cap H_j$ is infinite-dimensional, we can choose
$u\in S_H\cap Y\cap H_j\cap (P_jx)^\perp$.  Then $x+ru\in A_j$, but
\[
\|P_j(x+ru)\|
=
\|P_jx+ru\|
=
\sqrt{\|P_jx\|^2+r^2}
\ge r>n^{-1/2},
\]
a contradiction.  Thus $\varrho(A_j)\le n^{-1/2}$, and so
$\varrho(A_j)=n^{-1/2}$ for every $j$.  It follows that
\[
\Theta_H(n)\le n^{-1/2}.
\]

\smallskip
\noindent\emph{Lower bound.}
Let $0<\varepsilon<n^{-1/2}$. Then $n\varepsilon^2<1$, so
Corollary~\ref{cor:Hilbert-goal-iterates} gives
\[
[B_H]^n_{\varepsilon}=\sqrt{1-n\varepsilon^2}\,B_H\neq\varnothing.
\]
Hence, by the contrapositive of \eqref{eq:goal-empty}, we must have
\[
\varepsilon\le \Theta_H(n).
\]
Since this holds for every $\varepsilon<n^{-1/2}$, we obtain
\[
\Theta_H(n)\ge n^{-1/2}.
\]

Combining the two inequalities yields
\[
\Theta_H(n)=n^{-1/2}.
\]
\end{proof}

\section{The scalar $L_p(\mu)$-spaces}\label{sec:Lp-mu}

Let $(\Omega,\Sigma,\mu)$ be a $\sigma$-finite measure space and
$1\le p<\infty$. We assume throughout that $L_p(\mu)$ is
infinite-dimensional. For each $n\in\mathbb N$ we fix a partition
\[
\Omega=\bigsqcup_{k=1}^n E_k
\]
such that $L_p(E_k)$ is infinite-dimensional for every $k$. Such a
partition exists under the standing hypotheses: if $\mu$ has a non-atomic
part, split it into $n$ non-null measurable pieces; if $\mu$ is purely
atomic, then $L_p(\mu)$ being infinite-dimensional means that there are
infinitely many atoms, which can be distributed into $n$ infinite families.

\begin{proposition}\label{prop:Lp-mu-upper}
Let $1\le p<\infty$ and let $(\Omega,\Sigma,\mu)$ be a $\sigma$-finite
measure space such that $L_p(\mu)$ is infinite-dimensional. Then, for every
$n\in\mathbb{N}$,
\[
\Theta_{L_p(\mu)}(n)\ \le\ n^{-1/p}.
\]
\end{proposition}

\begin{proof}
Fix $n\in\mathbb N$, and choose a partition
\[
\Omega=\bigsqcup_{k=1}^n E_k
\]
such that $L_p(E_k)$ is infinite-dimensional for every $k$.

For each $k=1,\dots,n$ define the projection
\[
P_k f := f\cdot\mathbf{1}_{E_k}, \qquad f\in L_p(\mu).
\]
Then $P_k:L_p(\mu)\to L_p(\mu)$ is a contractive projection.

For $f\in L_p(\mu)$ we have
\[
\sum_{k=1}^n \|P_k f\|_p^p
= \sum_{k=1}^n \int_{E_k} |f|^p\, d\mu
= \int_{\Omega} |f|^p\, d\mu
= \|f\|_p^p.
\]
In particular, if $f\in B_{L_p(\mu)}$, then
\[
\sum_{k=1}^n \|P_k f\|_p^p \le 1.
\]

Define, for $1\le k\le n$,
\[
A_k := \bigl\{ f\in B_{L_p(\mu)}:\ \|P_k f\|_p\le n^{-1/p} \bigr\}.
\]
Each $A_k$ is closed, convex and bounded.

If $f\in B_{L_p(\mu)}$, then $\|P_k f\|_p^p \le 1/n$ for some $k$, that
is, $\|P_k f\|_p\le n^{-1/p}$, and hence $f\in A_k$. Thus
\[
B_{L_p(\mu)} = \bigcup_{k=1}^n A_k.
\]

\smallskip
\noindent\emph{Lower bound on $\varrho(A_k)$.}
If $\|g\|_p\le n^{-1/p}$, then $\|P_k g\|_p\le \|g\|_p\le n^{-1/p}$, so
$g\in A_k$. Taking $x=0$ and $Y=L_p(\mu)$ in the definition of the
essential inradius, we get
\[
n^{-1/p} B_{L_p(\mu)} \subset A_k,
\]
so $\varrho(A_k)\ge n^{-1/p}$.

\smallskip
\noindent\emph{Upper bound on $\varrho(A_k)$.}
Suppose, for a contradiction, that $\varrho(A_k)>n^{-1/p}$. Then there exist
$f\in A_k$, a finite-codimensional subspace $Y\le L_p(\mu)$, and
$r>n^{-1/p}$ such that
\[
f + r(B_{L_p(\mu)}\cap Y)\subset A_k.
\]
Let $v := P_k f\in L_p(E_k)$; then $\|v\|_p\le n^{-1/p}$.

Since $Y$ has finite codimension, there exist $\varphi_1,\dots,\varphi_m\in
L_p(\mu)^*$ such that
\[
Y = \bigcap_{i=1}^m \ker\varphi_i.
\]

Since $L_p(E_k)$ is infinite-dimensional and $(E_k,\Sigma|_{E_k},\mu|_{E_k})$
is $\sigma$-finite, there exists a sequence $(C_q)_{q\ge1}$ of pairwise
disjoint measurable subsets of $E_k$ such that
\[
0<\mu(C_q)<\infty\qquad(q\ge1).
\]
For $j\ge1$ and $0\le \ell\le m$, set
\[
B_{j,\ell}:=C_{(m+1)(j-1)+\ell+1},
\qquad
U_j:=\bigcup_{\ell=0}^m B_{j,\ell}.
\]
Then the sets $U_j$ are pairwise disjoint. Set
\[
a_j := \|v\mathbf 1_{U_j}\|_p.
\]
Since the $U_j$ are disjoint and $v\in L_p(E_k)$, we have
\[
\sum_{j=1}^\infty a_j^p
=
\sum_{j=1}^\infty \int_{U_j}|v|^p\,d\mu
\le \|v\|_p^p<\infty,
\]
hence $a_j\to0$.

For each $j$, define
\[
f_{j,\ell}:=\mathbf 1_{B_{j,\ell}},\qquad \ell=0,\dots,m,
\]
and consider the linear map
\[
T_j:\K^{m+1}\longrightarrow\K^m,\qquad
T_j(\alpha_0,\dots,\alpha_m)
=
\Bigl(\varphi_i\Bigl(\sum_{\ell=0}^m \alpha_\ell f_{j,\ell}\Bigr)\Bigr)_{i=1}^m.
\]
Since $\dim \K^{m+1}>\dim \K^m$, the kernel of $T_j$ is non-trivial, so
there exists a non-zero
\[
w_j=\sum_{\ell=0}^m \alpha_\ell f_{j,\ell}\in Y
\]
supported in $U_j$. After normalising, we may assume that
\[
u_j:=\frac{w_j}{\|w_j\|_p}
\]
satisfies $\|u_j\|_p=1$ and is supported in $U_j$.

Fix $r>n^{-1/p}$ as above and consider the continuous function
\[
F(a)
=
\|v\|_p^p - a^p + (r-a)^p,\qquad a\ge0.
\]
We have
\[
F(0)=\|v\|_p^p+r^p>r^p>1/n,
\]
so, by continuity, there exists $\delta>0$ such that
\[
0<\delta<r
\quad\text{and}\quad
F(a)>1/n \ \text{whenever}\ 0\le a\le\delta.
\]

Choose $j_0$ large enough so that $a_{j_0}\le\delta$, and set
\[
u:=u_{j_0},\qquad U:=U_{j_0}.
\]
Then $u\in Y$, $\|u\|_p=1$, and $\|v\mathbf 1_U\|_p=a_{j_0}\le\delta$.

Decompose
\[
v=v_1+v_2,
\qquad
v_1=v\mathbf 1_{E_k\setminus U},\quad
v_2=v\mathbf 1_U.
\]
Since $u$ is supported in $U$, we have
\[
\|v+ru\|_p^p
=
\|v_1\|_p^p+\|v_2+ru\|_p^p.
\]
By the triangle inequality in $L_p$,
\[
\|v_2+ru\|_p
\ge r-\|v_2\|_p
= r-a_{j_0},
\]
and therefore
\[
\|v+ru\|_p^p
\ge
\|v\|_p^p-a_{j_0}^p+(r-a_{j_0})^p
=
F(a_{j_0})
>
1/n.
\]

Finally, set $y:=ru\in Y$. Then $\|y\|_p=r$ and
\[
\|P_k(f+y)\|_p^p
=
\|v+ru\|_p^p
>
1/n,
\]
so $\|P_k(f+y)\|_p>n^{-1/p}$, and therefore $f+y\notin A_k$. This
contradicts the assumption that
\[
f+r(B_{L_p(\mu)}\cap Y)\subset A_k.
\]

Thus no such $r>n^{-1/p}$ exists and $\varrho(A_k)\le n^{-1/p}$. Combining
with the lower bound we obtain $\varrho(A_k)=n^{-1/p}$ for each $k$, and
hence
\[
\Theta_{L_p(\mu)}(n)
\le \max_{1\le k\le n}\varrho(A_k)
= n^{-1/p}.
\]
\end{proof}

\begin{remark}
Specialising to the counting measure on $\mathbb{N}$ we recover the same upper
bound for the sequence spaces $\ell_p$.
\end{remark}

Combining Proposition~\ref{prop:Lp-mu-upper} with Raja's general lower bound gives matching power-type estimates for scalar $L_p(\mu)$.

\begin{corollary}\label{cor:Lp-mu-sharp}
Let $1<p<\infty$ and let $(\Omega,\Sigma,\mu)$ be a $\sigma$-finite
measure space such that $L_p(\mu)$ is infinite-dimensional and admits an
equivalent $p$-AUS norm. Then there exists a constant $c_{p,\mu}>0$ such that
\[
\Theta_{L_p(\mu)}(n)\ \ge\ c_{p,\mu}\,n^{-1/p}
\qquad(n\ge1).
\]
In particular,
\[
\Theta_{L_p(\mu)}(n)\ \asymp\ n^{-1/p}.
\]
\end{corollary}

\begin{proof}
By Theorem~\ref{thm:Raja-lower}, there exists $c_{p,\mu}>0$ such that
\[
\Theta_{L_p(\mu)}(n)\ \ge\ c_{p,\mu}\,n^{-1/p}
\qquad(n\ge1),
\]
while Proposition~\ref{prop:Lp-mu-upper} supplies the matching upper bound.
\end{proof}

\section{Bochner spaces}\label{sec:Bochner}

The block-decomposition argument of Proposition~\ref{prop:Lp-mu-upper} carries
over verbatim to the Bochner space $L_p(\mu;E)$ provided the underlying
measure space is non-atomic, since the construction uses only the
$\ell_p$-additivity of the norm over disjoint supports and the ability to
partition measurable sets into pieces of small norm contribution.

\begin{proposition}\label{prop:Bochner-upper}
Let $1\le p<\infty$, let $(\Omega,\Sigma,\mu)$ be a non-atomic $\sigma$-finite
measure space, and let $E\neq\{0\}$ be a Banach space.  Then, for every
$n\in\mathbb{N}$,
\[
\Theta_{L_p(\mu;E)}(n)\ \le\ n^{-1/p}.
\]
\end{proposition}

\begin{proof}
Set $X:=L_p(\mu;E)$.  Choose a partition
$\Omega=\bigsqcup_{k=1}^n E_k$ into pairwise disjoint measurable sets of
positive measure. Since $\mu$ is non-atomic, each restricted measure space
$(E_k,\Sigma|_{E_k},\mu|_{E_k})$ is again non-atomic.

For $F\in X$, define
\[
P_k F := F\cdot\mathbf{1}_{E_k}\qquad(k=1,\dots,n).
\]
Each $P_k$ is a contractive projection on $X$, and
\[
\|F\|_p^p = \sum_{k=1}^n \|P_k F\|_p^p.
\]

Define
\[
A_k := \bigl\{F\in B_X:\ \|P_k F\|_p\le n^{-1/p}\bigr\}
\qquad(k=1,\dots,n).
\]
Each $A_k$ is closed, convex and bounded.  If $F\in B_X$, then
$\sum_{k=1}^n \|P_k F\|_p^p\le 1$, so for some $k$ we have
$\|P_k F\|_p^p\le 1/n$, hence $\|P_k F\|_p\le n^{-1/p}$ and $F\in A_k$.
Thus $B_X=\bigcup_{k=1}^n A_k$.

If $\|G\|_p\le n^{-1/p}$, then $\|P_k G\|_p\le\|G\|_p\le n^{-1/p}$, so
$G\in A_k$.  Hence $n^{-1/p}B_X\subset A_k$ and $\varrho(A_k)\ge n^{-1/p}$.

It remains to show that $\varrho(A_k)\le n^{-1/p}$.  Suppose, for a
contradiction, that there exist $F\in A_k$, a closed finite-codimensional
subspace $Y\le X$, and $r>n^{-1/p}$ such that
\[
F + r(B_X\cap Y)\subset A_k.
\]
Write $v:=P_k F\in L_p(E_k;E)$, so that $\|v\|_p\le n^{-1/p}$.

Choose $\varphi_1,\dots,\varphi_m\in X^*$ with
$Y=\bigcap_{i=1}^m\ker\varphi_i$.  Fix $e_0\in S_E$.

Because $(E_k,\Sigma|_{E_k},\mu|_{E_k})$ is non-atomic and $\sigma$-finite,
we may partition $E_k$ into a sequence of pairwise disjoint measurable sets
$(B_j)_{j\ge1}$ such that
\[
0<\mu(B_j)<\infty
\quad\text{and}\quad
a_j:=\|v\mathbf{1}_{B_j}\|_p\to 0.
\]

For each $j$, partition $B_j$ into $m+1$ pairwise disjoint measurable
subsets $B_{j,0},\dots,B_{j,m}$, each of positive measure, and define
\[
g_{j,\ell} := \mathbf{1}_{B_{j,\ell}}\,e_0\in X
\qquad(\ell=0,\dots,m).
\]
Consider the linear map
\[
T_j:\K^{m+1}\longrightarrow\K^m,\qquad
T_j(\alpha_0,\dots,\alpha_m)
=
\Bigl(\varphi_i\Bigl(\sum_{\ell=0}^m \alpha_\ell g_{j,\ell}\Bigr)\Bigr)_{i=1}^m.
\]
Since $\dim\K^{m+1}>\dim\K^m$, there exists a non-zero
\[
w_j = \sum_{\ell=0}^m \alpha_\ell\,g_{j,\ell}\in Y
\]
supported in $B_j$.  Set $u_j:=w_j/\|w_j\|_p$.  Then $u_j\in Y$,
$\|u_j\|_p=1$, and $\operatorname{supp} u_j\subset B_j$.

Define the continuous function
\[
\Psi(a) := \|v\|_p^p - a^p + (r-a)^p\qquad(a\ge0).
\]
Since $\Psi(0)=\|v\|_p^p+r^p>r^p>1/n$, continuity gives
$\delta\in(0,r)$ such that $\Psi(a)>1/n$ whenever $0\le a\le\delta$.
Choose $j$ so large that $a_j\le\delta$, and set $u:=u_j$, $B:=B_j$.

Write $v=v_1+v_2$ with $v_1:=v\mathbf{1}_{E_k\setminus B}$ and
$v_2:=v\mathbf{1}_B$.  Since $\operatorname{supp} u\subset B$, we have
\[
\|v+ru\|_p^p = \|v_1\|_p^p + \|v_2+ru\|_p^p.
\]
By the triangle inequality in $L_p(\mu;E)$,
\[
\|v_2+ru\|_p \ge r - \|v_2\|_p = r - a_j,
\]
so
\[
\|v+ru\|_p^p
\ge \|v\|_p^p - a_j^p + (r-a_j)^p = \Psi(a_j)>1/n.
\]
Therefore $\|P_k(F+ru)\|_p = \|v+ru\|_p > n^{-1/p}$, so $F+ru\notin A_k$.
But $u\in B_X\cap Y$, hence $F+ru\in F+r(B_X\cap Y)\subset A_k$, a
contradiction.

Thus $\varrho(A_k)=n^{-1/p}$ for every $k$, and
\[
\Theta_{L_p(\mu;E)}(n)\le n^{-1/p}.
\]
\end{proof}

\begin{remark}[A partial answer to Problem~4.7 of \cite{RajaCoveringIndex}]\label{rmk:Problem47}
In \cite[Problem~4.7]{RajaCoveringIndex} Raja asks whether the behaviour of the
covering index $\Theta_X(n)$ ``really depends'' on both the asymptotic moduli
of uniform convexity and uniform smoothness (AUC and AUS).
Proposition~\ref{prop:Bochner-upper} shows that the answer is partly negative,
at least on the level of upper bounds.

Fix $1<p<\infty$ and a non-atomic $\sigma$-finite measure space
$(\Omega,\Sigma,\mu)$. Let $E_1=\K$, so that
\[
L_p(\mu;E_1)=L_p(\mu).
\]
As recalled earlier (and used in the proof of
Corollary~\ref{cor:Lp-mu-sharp}), $L_p(\mu)$ admits an equivalent $p$-AUS
norm.

Now let $E_2$ be any separable Banach space which is not AUSable; for
instance, one may take a space of Szlenk index $>\omega$, as in
\cite{OdellSchlumprechtZsak2009} together with
\cite{KnaustOdellSchlumprecht1999,RajaAUS,LancienSzlenkSurvey}.
Choose $A\in\Sigma$ with $0<\mu(A)<\infty$. Then
\[
J_A\colon E_2\longrightarrow L_p(\mu;E_2),\qquad
J_A(x)=\mu(A)^{-1/p}\,x\,\mathbf 1_A,
\]
is an isometric embedding. Consequently, if $L_p(\mu;E_2)$ admitted an
equivalent AUS norm, then its restriction to $J_A(E_2)\cong E_2$ would give
an equivalent AUS norm on $E_2$, a contradiction. Thus $L_p(\mu;E_2)$ is
not AUSable.

Nevertheless, Proposition~\ref{prop:Bochner-upper} gives
\[
\Theta_{L_p(\mu;E_i)}(n)\le n^{-1/p}
\qquad(i=1,2,\ n\in\mathbb N).
\]
Thus spaces with very different asymptotic geometries can exhibit the same
power-type upper decay for the covering index, and one cannot hope for a
simple correspondence between the asymptotic moduli of $L_p(\mu;E)$ and the
precise shape of $\Theta_{L_p(\mu;E)}$.
\end{remark}

\section{Non-commutative $L_p$-spaces of von Neumann algebras}

We briefly indicate how Raja's lower bound extends to non-commutative $L_p$-spaces.  Throughout this section, $(M,\tau)$ denotes a semifinite von Neumann algebra equipped with a faithful normal semifinite trace~$\tau$.

We recall the standard non-commutative $L_p$-spaces; see Pisier--Xu \cite{PisierXu}.

\begin{definition}[Non-commutative $L_p$-spaces]
Let $1\le p<\infty$.  The non-commutative $L_p$-space $L_p(M,\tau)$ is defined as the completion of the set of $\tau$-measurable operators $x$ affiliated with $M$ for which
\[
\|x\|_p := \bigl(\tau(|x|^p)\bigr)^{1/p}<\infty.
\]
\end{definition}

For $1<p<\infty$, the spaces $L_p(M,\tau)$ are uniformly convex and uniformly smooth; non-commutative Clarkson inequalities (see \cite[Section~34.5]{PisierXu}) show that their modulus of smoothness is of power type
\[
r=\min\{p,2\}.
\]
In particular, $L_p(M,\tau)$ is asymptotically uniformly smooth of power type $r$ (see the appendix for details), and hence $r$-AUS.

Combining this with Theorem~\ref{thm:Raja-lower} yields:

\begin{proposition}\label{prop:non-comm-Lp-lower}
Let $(M,\tau)$ be a semifinite von Neumann algebra such that
$L_p(M,\tau)$ is infinite-dimensional, let $1<p<\infty$, and set
$r=\min\{p,2\}$.
Then there exists $c_{p}(M)>0$ such that
\[
\Theta_{L_p(M,\tau)}(n)\ \ge\ c_{p}(M)\,n^{-1/r}
\qquad(n\ge1).
\]
\end{proposition}

Thus the covering index of any non-commutative $L_p$-space decays at least like $n^{-1/r}$ with $r=\min\{p,2\}$.

\begin{remark}
If $M=L_\infty(\Omega,\mu)$ and $\tau(f)=\int_\Omega f\,d\mu$, then $L_p(M,\tau)$ is just $L_p(\mu)$ and Proposition~\ref{prop:non-comm-Lp-lower} yields a lower bound of order $n^{-1/r}$ with $r=\min\{p,2\}$.  Our explicit covering in the commutative case (Proposition~\ref{prop:Lp-mu-upper}) shows that the optimal exponent is at most $1/p$; when $L_p(\mu)$ moreover admits an equivalent $p$-AUS norm, Corollary~\ref{cor:Lp-mu-sharp} gives $\Theta_{L_p(\mu)}(n)\asymp n^{-1/p}$.  Thus the abstract non-commutative argument is not sharp even in the commutative setting when $p>2$.

If $M=B(H)$ with the canonical trace, then $L_p(M,\tau)$ is the Schatten class $S_p$ on $H$.  Proposition~\ref{prop:non-comm-Lp-lower} then shows that
\[
\Theta_{S_p}(n)\ \gtrsim\ n^{-1/r},
\qquad r=\min\{p,2\}.
\]
We do not pursue corresponding upper bounds in the non-commutative setting here.
\end{remark}

\appendix

\section{Non-commutative Clarkson inequalities and AUS for $L_p(M,\tau)$}

Let $(M,\tau)$ be a semifinite von Neumann algebra with a faithful normal
semifinite trace, and let $L_p(M,\tau)$ be the associated non-commutative
$L_p$-space as in \cite{PisierXu}.  Pisier and Xu prove that analogues of
the Clarkson inequalities hold in this setting; see in particular
\cite[Section~34.5]{PisierXu}.  For instance, for $1<p<\infty$, the space
$L_p(M,\tau)$ is uniformly convex and uniformly smooth, and its modulus of
smoothness satisfies
\[
   \rho_{L_p(M,\tau)}(t)\ \le\ C_p\,t^{r},
   \qquad r=\min\{p,2\},
\]
for all $t>0$, with a constant $C_p$ depending only on $p$ (and not on
$M$ or $\tau$).  Thus non-commutative $L_p$-spaces enjoy global uniform
smoothness of some power type $r>1$.

By Remark~\ref{rmrk:modulus}, any uniform smoothness estimate of power type for $\rho_X$ immediately yields an asymptotic uniform smoothness estimate of the same power type for $\overline{\rho}_X$.

Applying this observation to $X=L_p(M,\tau)$, we conclude that, for $1<p<\infty$,
\[
\overline{\rho}_{L_p(M,\tau)}(t)\ \le\ C_p\,t^{r},
\qquad r=\min\{p,2\},
\]
and hence $L_p(M,\tau)$ is asymptotically uniformly smooth of power type $r>1$.  In particular, the natural norm on $L_p(M,\tau)$ is $r$-AUS with $r=\min\{p,2\}$.  This is the input used in the proof of Proposition~\ref{prop:non-comm-Lp-lower}: once $L_p(M,\tau)$ is $r$-AUS, Raja's general theorem (Theorem~\ref{thm:Raja-lower}) yields a power-type lower bound
\[
\Theta_{L_p(M,\tau)}(n)\ \gtrsim\ n^{-1/r}.
\]

Determining the \emph{precise} optimal AUS exponent and the corresponding sharp covering-index decay rate for specific non-commutative $L_p$-spaces remains an interesting open problem, even for classical examples such as the Schatten classes $S_p=L_p(B(H),\mathrm{Tr})$.


\end{document}